\documentclass[10pt]{amsart}
\usepackage{amssymb,amsmath,amsthm,amsfonts,amsopn,url}
\usepackage{amscd,amssymb,amsopn,amsmath,amsthm,graphics,amsfonts,enumerate,verbatim,calc}
\usepackage[all]{xy}
\theoremstyle{plain}
\newtheorem{thm}{Theorem}[section]
\newtheorem*{mt*}{Main Theorem}
\newtheorem*{cj*}{Conjecture}
\newtheorem*{nt*}{Notations}
\newtheorem*{qs*}{Question}

\newtheorem{prop}[thm]{Proposition}
\newtheorem{lemma}[thm]{Lemma}
\newtheorem{cor}{Corollary}
\newtheorem{rem}{Remark}
\newtheorem{example}{Example}[section]
\theoremstyle{definition}
\newtheorem*{definition}{Definition}
\newcommand{\ideal}[1]{\mathfrak{#1}}

\newcommand{\m}{\ideal{m}}

\newcommand{\p}{\ideal{p}}

\newcommand{\func}[1]{\mathrm{#1} \,}

\newcommand{\Ass}{\func{Ass}}

\newcommand{\Hom}{\func{Hom}}
\newcommand{\Ann}{\func{Ann}}
\newcommand{\Char}{\func{Char}}
\newcommand{\V}{\func{V}}
\newcommand{\T}{\func{T}}
\newcommand{\Z}{\func{Z}}

\newcommand{\Spec}{\func{Spec}}
\newcommand{\Comp}{\func{Cf}}
\newcommand{\Fin}{\func{Fin}}

\newcommand{\NN}{{\mathbb N}}
\newcommand{\ZZ}{{\mathbb Z}}

\title[]{Smooth Algebra and Finiteness of the Set of Associated Primes of Local Cohomology Modules
}

\author[]{Rajsekhar Bhattacharyya
}

\address{Dinabandhu Andrews College, Garia, Kolkata 700084, India
}

\email{rbhattacharyya@gmail.com
}


\begin{document}

\begin{abstract}
In this article, we study the behaviour of smooth algebra $R$ over local Noetherian local ring $A$. At first, we observe that for every $f\in R$, $R_f$ has finite length in the category of $D(R,A)$-module if dimension of $A$ is zero. This extends the result of Theorem 2 of \cite{Ly3}. We use this fact to generalize the result of Theorem 4.1 of \cite{BBLSZ}, from the finiteness of the set of associated primes of local cohomology module to that of Lyubeznik functor. Finally, we introduce the definition of $\Sigma$-finite $D$-modulue for smooth algebra and we extend the result of Theorem 1.3 of \cite{Nu3} from polynomial and power series algebra to smooth algebra. Theorem 1.3 of \cite{Nu3} comes out as a partial answer to a question raised by Melvin Hochster. Thus, we extend the partial answer to the above question from polynomial and power series algebra to smooth algebra over an arbitrary Noetherian local ring.
\end{abstract}

\maketitle

\section{introduction}
Throughout this paper $A$ and $R$ always denote commutative Noetherian rings with unity. Let $M$ be an $R$-module and $I\subset R$ be an ideal. Then, for every integer $i\geq 0$, we denote the $i$-th local cohomology of $M$ with support in $I$ by $H^i_I (M)$.
Although, in general, the local cohomology modules are not finitely generated, but the set of associated primes of $H^i_I(R)$ is finite for certain regular rings. Huneke and Sharp proved this for characteristic $p>0$ \cite{HS}. Lyubeznik showed this finiteness property for regular local rings of equal characteristic zero and finitely generated regular algebras over a field of characteristic zero \cite{Ly1}. Recently, in \cite{BBLSZ}, the finiteness of local cohomology module is proved for smooth algebra over $\ZZ$ and $\ZZ_{p\ZZ}$ and from results of that paper, it is evident that smooth algebra can play a crucial role in the study of local cohomology modules, even if the dimension of the base ring is more than one. 

In studying the finiteness of local cohomology module, $D$-module plays an important role, see \cite{Ly1, Ly2, Ly3} and also \cite{BBLSZ}. One of the important property of smooth $A$-algebra $R$, which resemblances polynomial and power series algebra, is that the ring $D(R,A)$ of $A$-linear differential operators of $R$, commutes with the base change, see Lemma 2.1 of \cite{BBLSZ}. In this article, we study the behaviour of smooth algebra over local ring and in this context, we generalize certain the results of \cite{Nu2}, and \cite{Nu3} from the polynomial and power series algebra to smooth algebra over a Noetherian local ring. In this context, it is to be noted that Lemma 2.1 of \cite{BBLSZ} plays a very important role in this generalization. %Later we use this fact to generalize the result of Theorem 4.1 of \cite{BBLSZ}, from the finiteness of the set of associated primes of local cohomology to that of Lyubeznik functor. 

For a field $A$, and for $A$-algebra $R$, $R_f$ has finite length in the category of $D(R,A$-modules for every $f\in R$, if $R$ satisfies the conditions of Theorem 2 of \cite{Ly3}. In Lemma 3.2 of \cite{Nu2}, the result of Theorem 2 of \cite{Ly3}, is extended from polynomial or power series algebra over a field to those algebras over zero dimensional ring. In section 3, at first in Theorem 3.2, we extend the result of Theorem 2 of \cite{Ly3} further, from polynomial or power series algebra over a field to smooth algebra over zero dimensional ring. Later, in Theorem 3.5, we use this fact to generalize the result of Theorem 4.1 of \cite{BBLSZ} (see Corollary 1), from the finiteness of the set of associated primes of local cohomology to that of Lyubeznik functor.  

In section 4, we study the following question raised by Melvin Hochster (see, Question 1.1 of \cite{Nu3}): 

\begin{qs*}
Let $(A,\m,k)$ be a local ring  and $R$ be a flat extension with regular closed fibre. Is $\Ass_R H^0_{\m R}H^i_I(R)$ finite for every ideal $I\subset R$ and $i\in \NN?$ 
\end{qs*}

In Theorem 1.3 of \cite{Nu3}, it is proved that the answer is positive when $R$ is either a polynomial or a power series ring over a Noetherian local ring $A$ and for an ideal $I\subset R$ such that $\dim(A/I\cap A)\leq 1$. In doing so, definition of $\Sigma$-finite $D$-module is introduced for polynomial or power series algebra. From definition of smoothness it is immediate that every closed fibres are regular. So, we can expect to lift the result of Theorem 1.3 of \cite{Nu3} to smooth algebra. So, at first, we introduce the definition of $\Sigma$-finite $D$-module for smooth algebra and then in Theorem 4.4, we extend the result and we get the partial answer to the above question, when $R$ is a smooth $A$-algebra.
%%%%%%%%%%%%%%%%%%%%%%%%%%%%%%%%%%%%%%%%%%%%%%%%%%%%%%%%%%%%%%%%%%%%%%%%%%%%%%%%%%%%%%%%%%%%%%%%%%%%%%%%%%%%%%%%%%%%%%%%%%
%%%%%%%%%%%%%%%%%%%%%%%%%%%%%%%%%%%%%%%%%%%%%%%%%%%%%%%%%%%%%%%%%%%%%%%%%%%%%%%%%%%%%%%%%%%%%%%%%%%%%%%%%%%%%%%%%%%%%%%%%%%%%%%%%%%%%%
\section{basic results}

Here we briefly review some of the basic results which we need in the next sections.

\subsection{Lyubeznik Functor}

Lyubeznik Functor is introduced by Lyubeznik in \cite{Ly1}. Here, we account for a brief description of it: Let $R$ be a Noetherian ring and $M$ be an $R$-module. Let $\Z$ be a closed subset of $\Spec R$. We set $H^i_{\Z} (M)$ as the $i$-th local cohomology module of $M$ with support in $\Z$. We notice that $H^i_{\Z} (M)=H^i_I (M)$, for $\Z=\V(I)=\{P\in\Spec R: I\subset P\}$. For any two closed subsets of $\Spec R$, $\Z_1\subset \Z_2$, there is a long exact sequence of functors
$$\ldots\to H^i_{\Z_1}\to H^i_{\Z_2}\to H^i_{\Z_1/\Z_2}\to \ldots$$
We set $\T =\T_1\circ \dots\circ\T_t$, where every functor $\T_j$ is either $H^i_{\Z}$ for some closed subset $\Z$ of $\Spec R$ or the kernel or image (or cokernel) of some map in the above long exact sequence. $\T$ is known as Lyubeznik functor.

We note the following well known fact for Lyubeznik functor.

\begin{rem}
Let $S$ be faithfully flat over $R$. Then, for Lyubeznik functor $\T$, $\Ass_{S} \T(S)$ is a finite set if and only if $\Ass_{R} \T(R)$ is a finite set.\newline
\end{rem}

\subsection{Introduction to D-modules}

Consider Noetherian rings $A$ and $R$, such that $A\subset R$, and we denote by $D(R,A)$ the ring of $A$-linear differential operators of $R$. This is the subring of $\Hom_A(R,R)$ defined inductively as follows: The differential operators of order zero are the homomorphisms induced by multiplying by elements in $R$. An element $\theta\in\Hom_A(R,R)$ is a differential operator of order less than or equal to $k+1$ if $\theta\cdot r -r\cdot\theta$ is a differential operator of order less than or equal to $k$ for every $r\in R=\Hom_R(R,R)$. 

It is well known that if $M$ is a $D(R,A)$-module, then $M_f$ has the structure of a $D(R,A)$-module  such that, for every $f \in R$, the natural morphism $M\to M_f$ is a morphism of $D(R,A)$-modules. As a result of this, since $R$ is a $D(R,A)$-module, $\T(R)$ is also a $D(R,A)$-module (see, Examples $2.1$ in \cite{Ly1}).

If $R=A[x_1,\ldots,x_n]$ or $R=A[[x_1,\ldots,x_n]]$, then $\frac{1}{t!} \frac{\partial^t }{\partial x_i ^t}$ can be viewed as a differential operator on $R$ where the integer $t!$ is not invertible. By Theorem $16.11.2$ in \cite{EGA}, in each of these cases, $D(R,A)$ is the free $R$-module with basis $\frac{1}{{t_1}!} \frac{\partial^{t_1} }{\partial x_1 ^{t_1}}\ldots\frac{1}{{t_n}!} \frac{\partial^{t_n} }{\partial x_n ^{t_n}}$, where for each $i= 1$ to $n$, $t_i\in \NN$. Thus, it follows that for every $A$-algebra $B$, $D(R,A)\otimes_A B\cong D(R\otimes_A B, B)$. In particular, for an ideal $I\subset A$, we have the isomorphism $D(R,A)/ID(R,A)\cong D(R/IR,A/IA)$. Moreover, if $M$ is a $D(R,A)$-module, then $IM$ is a $D(R,A)$-submodule and the last isomorphism gives the structure of $M/IM$ as a $D(R,A)$-module. 

In \cite{Ly3}, we observe a definition of a subcategory, denoted by $C(R,A)$, which is the smallest subcategory of $D(R,A)$-modules that contains $R_f$ for all $f\in R$ and that is closed under subobjects, extensions, and quotients. In particular, the kernel, image, and cokernel of a morphism of $D(R,A)$-modules that belong to $C(R,A)$ are also objects in $C(R,A)$. 

\subsection{D-modules of Finite Length}

A $D(R,A)$-module $M$ is simple, if its only $D(R,A)$-submodules are $0$ and $M$. We say that a $D(R,A)$-module $M$ has finite length, if there is a strictly ascending chain of $D(R,A)$-modules, 
$0\subset M_0 \subset M_1\subset \ldots \subset M_h =M,$
called a composition series of $M$,
such that $M_{i+1}/M_i$ is a nonzero simple $D(R,A)$-module for every $i=0,\ldots, h$. Here, $h$ is independent of the filtration and it is called the length of $M$. Moreover, for every filtration, the composition factors, $M_{i+1}/M_i,$ are the same, up to permutation and isomorphism.

\begin{nt*}
If $M$ is a $D(R,A)$-module of finite length, we denote the set of its composition factors by $\Comp{M}$.
\end{nt*}

\begin{rem}\label{RemFinLen}
(a) If $M$ is a nonzero simple $D(R,A)$-module, then $M$ has only one associated prime. This is because for every prime ideal $P\subset \Ass_R M$, $H^0_P(M)$ is a nonzero $D(R,A)$-submodule of $M$ and so it is the whole of $M$. As a consequence, if $M$ is a $D(R,A)$-module of finite length, then $\Ass_R M\subset \bigcup_{N\in\Comp{M}} \Ass_R N,$ which is finite. 

(b) If $0\to M'\to M\to M''\to 0$  is a short exact sequence of $D(R,A)$-modules of finite length, then $\Comp{M}=\Comp{M'}\bigcup \Comp{M''}$.
\end{rem}

We recall that, for a field $A$, and for $A$-algebra $R$, $R_f$ has finite length in the category of $D(R,A)$-modules for every $f\in R$, if $R$ satisfies the conditions of Theorem 2 of \cite{Ly3}. In this context, we note the following result.

\begin{lemma}
(a) If $M$ is an object in $C(R,A)$, then $\T(M)$ is also an object in this subcategory; in particular, $\T(R)$ belongs to $C(R,A)$\newline
(b) For $A$-algebra $R$, assume that $R_f$ has finite length in the category of $D(R,A)$-modules for every $f \in R$ and $M$ is an object of $C(R,A)$, then $M$ has finite length as a $D(R,A)$-module. As a consequence, $\T(R)$ would also have finite length.
\end{lemma}

\begin{proof}
(a) Proof is similar to that of Lemma 5 of \cite{Ly3}.

(b) Proof is similar to that of Corollary 6 of \cite{Ly3}.
\end{proof}

\subsection{Smooth Algebra}

We recall the definition of regular algebras, (see \cite{CAofM}, page 249). Here, we call finitely presented regular algebra over the base ring as smooth algebras. More precisely: A ring $R$ is said to be smooth over $A$ if $R$ is a finitely presented and flat $A$-algebra, such that for each prime ideal $\p$ of $A$, the fiber $R_{\p}/\p R_{\p}$ is geometrically regular over $A_{\p}/\p A_{\p}$. We have already seen that if $R=A[x_1,\ldots,x_n]$ or $R=A[[x_1,\ldots,x_n]]$, then for an $A$-algebra $B$, it follows that $D(R,A)\otimes_A B= D(R\otimes_A B, B)$. Recently in Lemma 2.1 of \cite{BBLSZ}, this result is been extended to $R$ when it is smooth over $A$. 

Here we mention the known results about the behaviour of smooth algebra under base change of finite type.

\begin{lemma}
Let $R$ be a smooth $A$-algebra, and let $A\rightarrow \bar{A}$ be a base change of finite type. Then $R\otimes_A \bar{A}$ is also smooth over $\bar{A}$. 
\end{lemma}

\begin{proof}
If $R$ be a finitely presented $A$-algebra, then under the base change $A\rightarrow \bar{A}$, $R\otimes_A \bar{A}$ is also a finitely presented algebra over $\bar{A}$. Now for the rest of the proof, see Lemma 4 in page 253 of \cite{CAofM}.
\end{proof}

\subsection{Direct Limit of Regular Rings}

Here we state a known result regarding direct limit of regular rings. We also present its proof for completeness.

\begin{lemma}
Let $(B_i,f_{ij})$ be a directed system of local rings whose transition maps are local ring maps. If each $B_i$ is a regular local ring and $B_{\infty}= \lim\limits_{\to i} B_i$ is Noetherian, then $B_{\infty}$ is a regular local ring.
\end{lemma}

\begin{proof}
Let $\m\subset B_{\infty}$ be the maximal ideal; it is the direct limit of the maximal ideal $\m_i\subset B_i$. We proceed by induction on $d=\dim \m/\m^2$. If $d= 0$, then $B_{\infty}= B_{\infty}/\m$ is a field and $B_{\infty}$ is a regular local ring. If $d> 0$ pick an $x\in \m- \m^2$. For some $i$ we can find an $x_i\in \m_i$ mapping to $x$. Note that $B_{\infty}/xB_{\infty}=\lim\limits_{\to j\geq i}B_j/x_iB_j$ is a Noetherian local ring. Since $x_i$ is in the set of minimal generator of $\m_j$ we find that $B_j/x_iB_j$ is a regular local ring. Hence by induction we see that $B_{\infty}/xB_{\infty}$ is a regular local ring. Since each $B_i$ is a domain, $B_{\infty}$ is also a domain. Hence $x$ is a nonzero divisor and we conclude that $B_{\infty}$ is a regular local ring.
\end{proof}
%We mention the known results about the behaviour of smooth algebra under localization. For completeness, we present the proofs.

%\begin{lemma}
%If $R$ is a smooth $A$-algebra, then for any multiplicatively closed set $W$ of $A$, $R_W$ is a smooth $A_W$-algebra.
%\end{lemma}

%\begin{proof}
%Clearly $R_W$ is flat over $A_W$. Consider $P\in \Spec A_W$ with $P=\p A_W$ for some $\p\in \Spec A$. Now $\kappa(\p A_W)=(A_W)_{\p A_W}/(\p A_W)(A_W)_{\p A_W}= A_{\p}/\p A_{\p}=\kappa(\p)$. Thus for any finite extension $L$ of $\kappa(\p A_W)$, $R_W\otimes_{A_W} L= (R_W\otimes_{A_W} \kappa(\p A_W))\otimes_{\kappa(\p A_W)} L= ((R\otimes_A A_W)\otimes_{A_W} \kappa(\p A_W))\otimes_{\kappa(\p A_W)} L= R\otimes_A \kappa(\p)\otimes_{\kappa(\p)} L= R\otimes_A L$. Since $R\otimes_A L$ is regular, we conclude.
%\end{proof}
%%%%%%%%%%%%%%%%%%%%%%%%%%%%%%%%%%%%%%%%%%%%%%%%%%%%%%%%%%%%%%%%%%%%%%%%%%%%%%%%%%%%%%%%%%%%%%%%%%%%%%%%%%%%%%%%%%%%%%%%%%%%%%%%%%%%%%%%%%%%%%%%%%%%%%%%%%%%%%%%%%%%%%%%%%%%%%%%%%%%%%%%%%%%%%%%%%%%%%%%%%%%%%%%%%%%%%%%%%%%%%%%%%%%%%%%%%%%%%%%%%%%%%%%%%%%%%%%%%%%
\section{smooth algebra, results of \cite{Nu2} and lyubeznik functor}

In this section, at first in Theorem 3.2, we extend the result of Theorem 2 of \cite{Ly3}, from polynomial or power series algebra over a field to smooth algebra over zero dimensional ring. Later in Theorem 3.5, we use this fact to generalize the result of Theorem 4.1 of \cite{BBLSZ} (see Corollary 1), from the finiteness of the set of associated primes of local cohomology to that of Lyubeznik functor.

To explore our first main result, we need the following lemma. %true when For smooth algebra over zero dimensional ring, we have the following proposition. Proof is similar to that of Lemma 3.2 of \cite{Nu2}.

\begin{lemma}
Let $k$ be a field of characteristic $p> 0$ and $B$ be a finitely generated $k$-algebra. Set $B_i= B\otimes_k k^{1/p^i}$. If $B_i$ is regular for every $i\geq 0$, then $\lim\limits_{\to i} B_i= B_{\infty}= B\otimes_k k^{1/p^{\infty}}$ is regular.
\end{lemma}

\begin{proof}
Since $B$ is a finitely generated $k$-algebra, $B_i$ is also a finitely generated $k^{1/p^i}$-algebra due to base change. From flatness of $k^{1/p^i}\rightarrow k^{1/p^{i+1}}$, we find that $B_{i+1}$ is also a flat $B_i$-algebra. $B_{\infty}$ is also a finitely generated $k^{1/p^{\infty}}$ and so it is Noetherian and it turns out that $B_{\infty}= \lim_{\rightarrow i} B_i$. For a maximal ideal $\m\subset B_{\infty}$, let $\m_i= \m\cap B_i$ be a prime ideal of $B_i$. Consider the local rings $(B_i)_{\m_i}$. Observe the following commutative diagram 
\[
\CD
\ldots @>>>(B_{i})_{\m_{i}} @>>>(B_{i+1})_{\m_{i+1}} @>>>\ldots @>>>(B_{\infty})_{\m}\\
@. @AAA @AAA @. @AAA @. \\
\ldots @>>>B_{i} @>>>B_{i+1} @>>>\ldots @>>>B_{\infty}
\endCD
\]

where each map is a localization map and each of the horizontal map of local rings is injective, since each of the local ring is faithfully flat extension of the previous. For any $b/s\in (B_{\infty})_{\m}$, take $b,s \in B_{\infty}$, such that $s\notin \m$. Clearly, there exists, some $B_i$, where $b,s\in B_i$ and $s\notin \m_i$. Thus $b/s\in (B_i)_{\m_i}$. This gives that $(B_{\infty})_{\m}= \lim\limits_{\to i}(B_i)_{\m_i}$ and it is Noetherian. So from Lemma 2.3 we find that $(B_{\infty})_{\m}$ is a regular local ring. Thus $B_{\infty}$ is a regular ring.
\end{proof}

Now we state our first main result of this section.

\begin{thm}
Let $A$ be a zero-dimensional Noetherian ring. Let $R$ be a smooth $A$-algebra containing $A$. Then, $R_f$ has finite length as a $D(R,A)$-module for every $f\in R$.
\end{thm}

\begin{proof}
Here, $A$ has finite length as an $A$-module, and let $0=N_0\subset N_1\subset \ldots \subset N_\ell=A$ be a finite filtration of ideals such that $N_{j+1}/N_j$ is isomorphic to a field. Since $A$ sits inside $R$, we have an induced filtration of $D(R,A)$-modules, $0=N_0 R_f\subset N_1 R_f \subset \ldots \subset N_\ell R_f=R_f$. It suffices to prove that $N_{j+1} R_f/N_j R_f$ is of finite length for $j=1,\ldots, \ell$. We note that $N_{j+1} R_f/N_j R_f$ is zero or isomorphic to $ (R/\m R)_f$ for some maximal ideal $\m\subset A$. 

Set $k= A/\m$ and $B= R/\m R$ which is a finitely generated $k$-algebra. Since $R$ is smooth over $A$, due to base change, $R/\m R$ is also smooth over $k$ and hence it is regular. If $\Char{k}= 0$, then using (a) of Theorem 2 of \cite{Ly3} we can say that $(R/\m R)_f$ is of finite length as $D(R/\m R,A/\m)$ module. If $\Char{k}= p> 0$, set $B_i= R/\m R\otimes_k k^{1/p^i}$ and due to smoothness $B_i$ is regular for each $i$. Thus from Lemma 3.1, we get $B_{\infty}= B\otimes_k k^{1/p^{\infty}}= R/\m R\otimes_k k^{1/p^{\infty}}$ is regular. Thus using (c) of Theorem 2 of \cite{Ly3}, again we find that $(R/\m R)_f$ is of finite length as $D(R/\m R,A/\m)$ module. Thus, irrespective of the characteristics of $k= A/\m$, $(R/\m R)_f$ is of finite length as $D(R/\m R,A/\m)$ module.

So, $N_{j+1} R_f/N_j R_f$ has finite length as a $D(R/m R,A/m A)$-module and from Lemma 2.1 of \cite{BBLSZ}, we get $D(R/\m R,A/\m A)= D(R,A)/\m D(R,A)$. Thus, it has finite length as a $D(R,A)$-module, which concludes the proof. 
\end{proof}

The following results generalizes that of Proposition 3.3 of \cite{Nu2} from polynomial or power series ring to smooth algebra over zero dimensional ring.
 
\begin{prop}
Let $A$ be a zero-dimensional commutative Noetherian ring. Let $R$ be a smooth $A$-algebra such that $A\subset R$.
Then, $\Ass_R M$ is finite for every object in  $M\in C(R,A)$; in particular, this holds for $\T (R)$ for Lyubeznik functor $\T$.  
\end{prop}

\begin{proof}
By Theorem 3.2, $R_f$ has finite length in the category of $D(R,A)$-modules for every $f \in R$. If $M$ is an object of $C(R,A)$, then by Lemma 2.1, $M$ has finite length as a $D(R,A)$-module. Now, using Lemma 3.1 of \cite{Nu2}, we conclude.
\end{proof}

%\begin{lemma}
%Let $A$ be a one-dimensional ring,  $\pi\in A$ be an element such that $\dim(A/\pi A)=0,$ and let $R$ be a smooth $A$-algebra such that $R$ faithfully flat over $A$. Then, $R_f/\pi R_f$ has finite length as a $D(R,A)$-module for every $f\in R$.
%\end{lemma}

%\begin{proof}
%Being smooth, by Lemma 2.1 of \cite{BBLSZ}, the length of $R_f/\pi R_f$ as a $D(R,A)$-module or as a $D(R/\pi R, A/\pi A)$-module is the same. Now, $A/\pi A$ has dimension zero and $R/\pi R$ is also smooth over $A/\pi A$ due to base change of smooth algebra and also faithfully flat i.e. injective. Thus the result follows from  Lemma 3.2.
%\end{proof}

We need the following important result.

\begin{lemma}\label{LemmaAQFL}
%Let $A$ be a one-dimensional ring, $\pi\in A$ be an element such that $\dim(A/\pi A)=0,$ and 
Let $A$ be a Noetherian ring and $R$ be a smooth $A$-algebra. Consider an element $\pi\in A$ and set $\bar{A}$ and $\bar{R}$ denote $A/\pi A$ and $R/\pi R$ respectively. Let $M$ be a $D(R,A)$-module, such that $\Ann_{M} (\pi)$  and $M \otimes_R \bar{R}$ are objects in $C(\bar{R},\bar{A})$. Then, $\Ann_{\T(M)} (\pi)$  and $\T(M) \otimes_R \bar{R}$ are objects in $C(\bar{R},\bar{A})$ for every functor $\T$.
\end{lemma}

\begin{proof}
We first note the following: Since $M$ is $D(R,A)$-module, both $\Ann_{M} (\pi)$ and $M \otimes_R \bar{R}$ are $D(R,A)/\pi D(R,A)$-module. Since $R$ is smooth over$A$, from Lemma 2.1 of \cite{BBLSZ} we find that both $\Ann_{M} (\pi)$ and $M \otimes_R \bar{R}$ are $D(\bar{R},\bar{A})$-module. 

Now, rest of the proof is similar to that of Lemma 3.5 of \cite{Nu2}.

\end{proof}

Now, we state the next main results of this section in the following theorem and in the corollaries.

\begin{thm} \label{PropAssDO}
Let $A$ be a one-dimensional domain,% ring,  
and $\pi\in A$ be an element. Let $R$ be a faithfully flat smooth $A$-algebra. Then, the set of associated primes over $R$ of $\T (R)$ that contain $\pi$ is finite, for Lyubeznik functor $\T$.
\end{thm}

\begin{proof}
Set $\bar{A}= A/\pi A$ and $\bar{R}= R/\pi R$. The set of associated primes of $\T (R)$ that contain $\pi$ is $\Ass_R \Ann_{\T(R)}  (\pi)$. Here $\pi\in A$ is a nonzero divisor of $A$. Since $R$ is flat over $A$, $\pi\in A$ is also a nonzero divisor of $R$. Thus $\Ann_{R}(\pi)= 0$. Moreover, $\bar{R}\in C(\bar{R},\bar{A})$. So, we can apply Lemma 3.4, to see that $\Ann_{\T(R)} (\pi)$ is in $C(\bar{R},\bar{A})$. 

Since $R$ is smooth over $A$, due to finite base change $A\rightarrow \bar{A}$, $\bar{R}$ is also smooth over $\bar{A}$. Here $\dim{\bar{A}}= 0$. Thus from Theorem 3.2 we get that $\bar{R}_{\bar{f}}$ is of finite length for every $\bar{f}\in \bar{R}$. From first paragraph of the proof, we already know that $\Ann_{\T(R)} (\pi)$ is in $C(\bar{R},\bar{A})$. Thus from Lemma 2.1, we conclude that $\Ann_{\T(R)} (\pi)$ is also a $D(\bar{R},\bar{A})$-module of finite length as an $\bar{R}$-module. So, we have that $\Ass_{\bar{R}} \Ann_{\T(R)} (\pi)$ is finite by Lemma 3.1 \cite{Nu2}. Thus $\Ass_{{R}}\Ann_{\T(R)} (\pi)$ is finite by Proposition 2.2 of \cite{Pu1}.
\end{proof}

Recall the definition of $p$-ring from \cite{CRTofM}, which is also referred as DVR of mixed characteristic (see, statement of Theorem 4.1 of \cite{BBLSZ}). Now, we have the result which generalizes that of Theorem 4.1 of \cite{BBLSZ}.

\begin{cor}
Let $A$ be a one dimensional regular domain containing $\ZZ$ and $R$ be a smooth local $A$-algebra of mixed characteristic $p> 0$, containing $A$. Then for Lyubeznik functor $\T$, $\Ass_R \T(R)$ is a finite set.\newline
\end{cor}

\begin{proof}
Due to smoothness, $R$ is regular local ring. Consider the prime $p\in \ZZ$. Set $\Ass_R \T(R)=S_1 \cup S_2$, where $\p\in S_1$ if and only if $p\in \p$ and $\p\in S_2$ if and only if $p\notin \p$. From Theorem 3.5 above, $S_1$ is finite, while from Corollary 4.5 of \cite{Nu1}, $S_2$ is finite. Thus $\Ass_R \T(R)$ is a finite set.
\end{proof}

\begin{rem}
The result of above corollary generalizes the result of Theorem 4.1 of \cite{BBLSZ}, from the finiteness of the set of associated primes of local cohomology to that of Lyubeznik functor. %It also generalizes the situation, where the base ring $A$ of Theorem 4.3 of \cite{BBLSZ} is now lifted to any arbitrary one dimensional domain. 
The technique of proof which is adopted here, is different from that of \cite{BBLSZ}. Moreover, from the technique of proof of Theorem 4.1 of \cite{BBLSZ}, it appears that, it may not be easy to generalize the result of Theorem 4.1 of \cite{BBLSZ} for Lyubeznik functor. 
\end{rem}

\begin{rem}
The result of above corollary is true for local case and till now we donot know whether it can be extended to an arbitrary smooth non local algebra. But recently, in Theorem F of \cite{Nu4}, similar result is obtained for polynomial ring over $\ZZ$.
\end{rem}

From structure of complete regular local ring we know that for every complete ramified regular local ring can be thought as an extension of a complete $p$-ring, see Theorem 29.7 and 29.8 of \cite{CRTofM}. The extension is the composition of a power series extension followed by a module finite `Eisenstein extension'. In the following corollary, we explore a special situation of such extension.

\begin{cor}
For complete regular local ring $R$ of mixed characteristic $p> 0$, if it is smooth over a complete $p$-ring, or more generally, for any regular local ring $R$ of mixed characteristic $p$, if its completion is smooth over a complete $p$-ring then for Lyubeznik functor $\T$, $\Ass_R \T(R)$ is a finite set.\newline
\end{cor}

\begin{proof}
Due to Remark 1 of the section 2, second assertion follows from first assertion and first assertion is immediate from previous corollary.
\end{proof}

\begin{rem}
Unramified complete regular local ring in mixed characteristic $p> 0$ is actually a power series ring over a complete $p$-ring. Power series rings are examples of smooth algebra. The results in the above corollary include those complete ramified regular local rings which are not power series rings, but which are smooth over complete $p$-rings.
\end{rem} 

\begin{example}
For example of smooth regular local ring over DVR of mixed characteristic, see Theorem 3.4 of \cite{Du} and also observe Corollary 6 of \cite{GL}.
\end{example}
%%%%%%%%%%%%%%%%%%%%%%%%%%%%%%%%%%%%%%%%%%%%%%%%%%%%%%%%%%%%%%%%%%%%%%%%%%%%%%%%%%%%%%%%%%%%%%%%%%%%%%%%%%%%%%%%%%%%%%%%%%%%%%%%%%%%%%
%%%%%%%%%%%%%%%%%%%%%%%%%%%%%%%%%%%%%%%%%%%%%%%%%%%%%%%%%%%%%%%%%%%%%%%%%%%%%%%%%%%%%%%%%%%%%%%%%%%%%%%%%%%%%%%%%%%%%%%%%%%%%%%%%%%%%%%%%%%%%%%%%%%%%%%%%%%%%%%%%%%%%%%%%%%%%%%%%%%%%%%%%%%%%%%%%%%%%%%%%%%%%%%%%%%%%%%%%%%%%%%%%%%%%%%%%%%%%%%%%%%%%%%%%%%%%%%%%%%%
\section{smooth algebra and $\Sigma$-finite $D$-module}

In this section, at first, we introduce the definition of $\Sigma$-finite $D$-module for smooth algebra and then, in Theorem 4.4, we extend the result of Theorem 1.3 of \cite{Nu3} from polynomial and power series algebra to smooth algebra over an arbitrary Noetherian local ring. As mentioned in the introduction that Theorem 1.3 of \cite{Nu3} comes out as a partial answer to a question raised by Melvin Hochster. Thus, we extend the partial answer to the above question from polynomial and power series algebra to smooth algebra over an arbitrary Noetherian local ring.

We begin with the following lemma. 

\begin{lemma}
Let $(A,\m,k)$ be a Noetherian local ring and $R$ be a smooth algebra over $A$. Then, 

(a) For every $f\in R$, $(R/\m R)_f$ is a $D(R/\m R,A/\m)$-module of finite length and it is a $D(R,A)$-module of finite length. 

(b) For every objects $M$ of $C(R/\m R,A/\m)$, $M$ has finite length as a $D(R/\m R,A/\m)$-module and hence it has finite length as a $D(R,A)$-module. %As a consequence, $\T(R)$ would also have finite length.
\end{lemma}

\begin{proof}
(a) From the part of the proof of Theorem 3.2, it follows that $(R/\m R)_f$ is a $D(R/\m R,A/\m)$-module of finite length. Since $(R/\m R)_f$ is also $D(R,A)/\m D(R,A)$-module, from Lemma 2.1 of \cite{BBLSZ}, we get that it is also a $D(R,A)$-module of finite length.

(b) From the result of (a), and using (b) of Lemma 2.1, it follows that $M$ has finite length as a $D(R/\m R,A/\m)$-module. From Lemma 2.1 of \cite{BBLSZ}, it follows that $M$ has finite length as a $D(R,A)/\m D(R,A)$-module. Hence, it has finite length as a $D(R,A)$-module. 
\end{proof}

Thus above lemma helps us to extend the definition of $\Sigma$-Finite module \cite{Nu3} for smooth algebras. 

\begin{definition}
Let $(A,\m,k)$ be a Noetherian local ring and $R$ be a smooth algebra over $A$. Let $M$ be a $D(R,A)$-module supported at $\m R$. We set $\Fin{M}$ as the set of all $D(R,A)$-submodules of $M$ that have finite length.
We say that  $M$ is $\Sigma$-Finite  if: 
\begin{itemize}
\item[(i)] $\bigcup_{N\in\Fin{M}} N=M,$
\item[(ii)] $\bigcup_{N\in\Fin{M}}\Comp{N}$ is finite, and
\item[(iii)] For every $N\in\Fin{M}$ and $L\in \Comp{N}$, $L\in C(R/\m R, A/\m)$.
\end{itemize}
We denote the set of composition factors of $M$, $\bigcup_{N\in\Fin{M}}\Comp{N}$, by $\Comp{M}$.
\end{definition}

\begin{rem}
In context to the above situation,
$$\Ass_R M\subset \bigcup_{N\in\Comp{M}} \Ass_R M$$
for every   $\Sigma$-Finite  
$D(R,A)$-module, $M$. In particular, $\Ass_R M$ is finite.
\end{rem}

The way which leads us to Theorem 4.4, from the definition of $\Sigma$-finite $D(R,A)$-module, is similar to that of \cite{Nu3}. But for the sake of completeness, we present the way. For this purpose, at first, we observe the following properties of the $\Sigma$-Finite $D(R,A)$-module.% whose proofs are similar to those of  Lemma 4.2, Proposition 4.3, corollary 4.4, (Lemma 3.5 + Proposition 3.6), Proposition 3.11

\begin{prop}
Consider smooth $A$-algebra $R$, over the local ring $(A,\m,k)$.

(a) Let $0\to M'\to M\to M''\to 0$ be a short exact sequence of $D(R,A)$-modules. If $M$ is $\Sigma$-Finite then $M'$ and $M''$ are $\Sigma$-Finite. Moreover, $\Comp{M}=\Comp{M'}\cup \Comp{M'}.$

(b) Let  $M$  and $M'$ be $\Sigma$-Finite $D(R,A)$-modules. Then, $M\oplus M'$ is also $\Sigma$-Finite.

(c) Let  $M$ be a $\Sigma$-Finite $D(R,A)$-module. Then, $H^i_I(M)$ is $\Sigma$-Finite for every ideal $I\subset R$ and $i\in\NN$.

(d) Let  $M_t$ be an inductive direct system of $\Sigma$-Finite $D(R,A)$-modules. If $\bigcup_t\Comp{M_t}$ is finite, then $\lim\limits_{\to t} M_t$ is $\Sigma$-Finite and $\Comp{M}\subset\bigcup_t\Comp{M_t}.$ 
\end{prop}

\begin{proof}
(a) Proof is similar to that of Proposition 3.6 of \cite{Nu3}.

(b) Proof is similar to that of Lemma 3.9 of \cite{Nu3}.

(c) Proof is similar to that of Corollary 3.10 of \cite{Nu3}.

(d) Proof is similar to that of Proposition 3.11 of \cite{Nu3}.
\end{proof}

\begin{prop}
Consider smooth $A$-algebra $R$, over the local ring $(A,\m,k)$. 

(a) Let $J\subset R$ be an ideal and $M$ be an $A$-module of finite length. Then, $H^i_{J}(M\otimes_A R)$ is a $D(R,A)$-module of finite length. Moreover, $\Comp{H^i_{JR}(M\otimes_A R)}\subset \bigcup_{j}\Comp{H^j_{JR}(R/\m R )}$. 

(b) Let $I\subset R$ be an ideal containing $\m R$. Then $H^i_I(R)$ is $\Sigma$-Finite for every $i\in\NN.$ 

(c) Let $I\subset R$ be an ideal containing $\m R$ and $J_1,\ldots, J_\ell\subset R$ be any ideals. Then $H^{j_1}_{J_1}\cdots H^{j_\ell}_{J_\ell}H^i_I(R)$ is $\Sigma$-Finite.
\end{prop}

\begin{proof}
(a) It is to be noted that for smooth $A$-algebra $R$, due to (c) of Theorem 2 and Corollary 6 of \cite{Ly3} and Lemma 2.1 of \cite{BBLSZ}, we find that $H^j_{JR}(R/\m R )$ is of finite length. Rest of the proof of this part is similar to that of Lemma 4.2 of \cite{Nu3}. 

(b) The proof is similar to that of Proposition 4.3 of \cite{Nu3}.

(c) The proof is immediate from (c) of Proposition 4.2 and (b) of Proposition 4.3.
\end{proof}

Now we state the main result of this section.

\begin{thm}
Let $(A,\m,K)$ be any Noetherian local ring. Let $R$ be a smooth $A$-algebra. Then, $\Ass_R H^0_{\m R}H^i_I(R)$ is finite for every ideal $I\subset R$ such that $\dim A/I\cap A\leq 1$ and every $i\in \NN.$ Moreover, if $\m R\subset \sqrt{I}$, $\Ass_R H^{j_1}_{J_1}\cdots H^{j_\ell}_{J_\ell} H^i_I(R)$ is finite for all ideals $J_1,\ldots, J_\ell\subset R$ and integers $j_1,\ldots,j_\ell\in\NN.$
\end{thm}

\begin{proof}
For the first assertion, use the results of (a) Proposition 4.2 and (b) of Proposition 4.3, and the proof follows in the similar manner of that of Proposition 4.5 of \cite{Nu3}.

Second assertion follows immediately from (c) of Proposition 4.3 and Remark 5. 
\end{proof}

%{\textbf{Acknowledgement:}}\newline
 
%I would like to thank Tony J. Puthenpurakal for his invaluable comments and suggestions.
%%%%%%%%%%%%%%%%%%%%%%%%%%%%%%%%%%%%%%%%%%%%%%%%%%%%%%%%%%%%%%%%%%%%%%%%%%%%%%%%%%%%%%%%%%%%%%%%%%%%%%%%%%%%%%%%%%%%%%%%%%%%%%%%%%%%%%
%%%%%%%%%%%%%%%%%%%%%%%%%%%%%%%%%%%%%%%%%%%%%%%%%%%%%%%%%%%%%%%%%%%%%%%%%%%%%%%%%%%%%%%%%%%%%%%%%%%%%%%%%%%%%%%%%%%%%%%%%%%%%%%%%%%%%%%%%%%%%%%%%%%%%%%%%%%%%%%%%%%%%%%%%%%%%%%%%%%%%%%%%%%%%%%%%%%%%%%%%%%%%%%%%%%%%%%%%%%%%%%%%%%%%%%%%%%%%%%%%%%%%%%%%%%%%%%%%%%%

\end{document}